\documentclass[a4paper,12pt]{article}
\usepackage{hyperref}
\usepackage{amsmath}
\usepackage{amssymb}
\usepackage{amsthm}
\usepackage{mathrsfs}
\newtheorem{theorem}{Theorem}[section]

\newtheorem{proposition}[theorem]{Proposition}

\newtheorem{remark}[theorem]{Remark}

\newcommand{\EE}{\mathbb{E}}

\newcommand{\NN}{\mathbb{N}}
  
\newcommand{\PP}{\mathbb{P}}

\newcommand{\wH}{{\widehat{H}}}
\newcommand{\bH}{{\cal{H}}}

\newcommand{\aL}{{\cal L}}

\newcommand{\un}{{\underline{n}}}
\newcommand{\on}{{\overline{n}}}

\newcommand{\E}{\mathbb{E}}

\usepackage{dsfont}

\begin{document}


\title{\itshape One step entropy variation in sequential sampling of species for the 
Poisson-Dirichlet Process}

\author{Servet Mart\'inez%
\thanks{E-mail address:  \texttt{smartine@dim.uchile.cl}.}
\qquad
Javier Santib\'a\~nez%
\thanks{E-mail address:  \texttt{jsantibanez@dim.uchile.cl}.}
\\
\parbox{5in}{%
Departamento de Ingenier{\'\i}a Matem\'atica and Centro de Modelamiento
Matem\'atico, UMI 2071 CNRS-UCHILE, Facultad de Ciencias F\'isicas y 
Matem\'aticas, Universidad de Chile, 
Santiago, Chile. 
}
}

\maketitle

\begin{abstract}
We consider the sequential sampling of species, where observed samples are
classified into the species they belong to.
We are particularly interested in studying some quantities describing the
sampling process when there is a new species discovery.
We assume that the observations  and species are organized as a two-parameter
Poisson-Dirichlet Process, which is commonly used as a Bayesian prior
in the context of entropy estimation, and we use
the computation of the mean posterior entropy
given a sample developed  in \cite{archer2014bayesian}.
Our main result shows the existence of a monotone functional,
constructed from the difference between the maximal
entropy and the mean entropy throughout the sampling process.
We show that this functional remains constant only when  a new
species discovery occurs.
\end{abstract}

\bigskip

\noindent {\bf AMS Classification Number:} 94A17

\bigskip 



\section{Introduction}

Consider the sequential sampling of species, where one takes a random sample
from a population and classifies each observation according to the
species (or classes) to which they belong.
Because the population is large, there are some rare species that may not
be observed.
We intend to understand and model the
discovery of a new species in this context
and to study related informational quantities.
Our main result shows that the two-step variation of differences
between the maximal entropy and the entropy allows
us to describe when a new species is discovered
in the Poisson-Dirichlet Process (PDP).
It is worth mentioning that our work is purely statistical.

\medskip

\noindent The two parameter PDP ---introduced by Pitman and Yor in 1997
\cite{pitman1997two}--- supplies random partitions with an
infinite number of components in $[0,1]$ and serves to model
the process of sampling species and
the times at which  new species are discovered, see
\cite{lmp}, \cite{gnedinp} and \cite{huillet2005unordered}.
This process has been used in ecology, but also in genetic applications
\cite{favaro2012new},
natural language processing  \cite{sharif2008overview} and finance
\cite{sosnovskiy2015financial}. In Section \ref{sec:PDP}, we will
introduce the PDP and some of the basic
properties that we shall use.

\medskip

\noindent Entropy is a way to measure the diversity
of communities in a sample and
our work focuses on studying some aspects of
the posterior entropy of the process of sampling
species in the PDP. The computation of posterior entropy
relies on the fact that given the sample from a PDP,
the posterior distribution is a mixture of a finite
Dirichlet distribution and a PDP.

\medskip

\noindent Much  of this paper concern with Bayesian entropy
 estimation, is due to the results in \cite{archer2014bayesian}, in which
the prior and  posterior mean entropies for the PDP were computed and
some of their properties  stated.
This is discussed in Section \ref{sec:entropia}.
In Proposition \ref{prop:entropy-post-PDP-desigualdad},
we provide  lower and upper bounds for
the entropy when the sample size is fixed.

\medskip

\noindent The main purpose of this work is to obtain an increasing
functional along the process constructed with posterior mean entropy
between two successive steps of the PDP with parameters $(\alpha,\theta)$.
This functional is,
\begin{equation}
\label{eq111}
\aL_{\ell}=
(\theta+\ell)({\wH}^{\hbox{\footnotesize{max}}}_{\ell}-{\wH}_{\ell}),
\end{equation}
and satisfies the monotone property $\aL_{\ell+1}\ge \aL_\ell$.
Here ${\wH}_\ell$ denotes the posterior entropy when observing a sample at
step $\ell$ and ${\wH}^{\hbox{\footnotesize{max}}}_{\ell}$ is its maximum
over all samples of size $\ell$.
Our main result is Theorem \ref{teo:Delta-ppal} in Section
\ref{sec:variacion}, where we show  that $\aL_\ell$ is increasing and
the equality $\aL_{\ell+1}=\aL_\ell$ is attained only when a new species is
discovered.

\medskip

\noindent We also show that the weighted
difference of entropies satisfies
$$
(\theta+\ell+1){\wH}_{\ell+1}-(\theta+\ell){\wH}_{\ell}> 0.
$$
The expression (\ref{nlogn2}) obtained in Theorem
\ref{teo:Delta-ppal}, for the above difference of weighted entropies,
allows us to think of the entropy as a sum of the `discovery values' of the
sampled species, plus an additive deterministic term depending on 
$\ell, \alpha$ and $\theta$.
On the other hand, the expression (\ref{nuevoDelta}) allows
us to write straightforwardly the functional $\aL_{\ell}$ as a sum of
positive rewards for `reinforcing the knowledge' of what it is known, and 
no additional additive term is required.
The discovery values and the reinforcement rewards
are expressed in terms of the digamma function.
This is discussed in Remark \ref{entdisc}. 

\medskip

\noindent We also study similar quantities in the frequentist
framework and relations in the same vein are shown in Proposition
\ref{freqc}.

\section{Poisson-Dirichlet Process}
\label{sec:PDP}

\noindent This section is devoted to the definition of the 
PDP and to supply some of its 
properties. We follow the articles
\cite{orbanz2014lecture}, \cite{buntine2012bayesian},
\cite{teh2010dirichlet}, \cite{orbanz2010bayesian}, 
\cite{sharif2008overview} and \cite{archer2014bayesian}. Since this is 
a well-known theory we only state
those results directly related to our work.

\medskip

\noindent Let $0\leq \alpha< 1$ and $\theta>-\alpha$. Consider 
independent random variables 
$\beta_k\sim \, \text{Beta}(1-\alpha, \theta +\alpha k)$. 
Let $\pi=(\pi_k: k\ge 1)$ be given by the two-parameter 
Griffiths-Engen-McCloskey
distribution, $GEM(\alpha, \theta)$,
\begin{equation*}
\pi_1:= \beta_1,\quad \pi_k := \beta_k \prod_{j=1}^{k-1}(1-\beta_j) \quad
k\geq 2,
\end{equation*}
which defines a probability vector a.s. Now consider a non-atomic probability 
measure $G$ defined on space ${\cal X}$. 
Let $(\phi_k: k\ge 1)$ be an i.i.d. sequence 
with distribution as $G$, then are all different a.s.
We assume $\phi=(\phi_k: k\ge 1)$ are independent of $\pi$. 
The discrete random measure 
\begin{equation}
\label{eqn:PDP-Theta}
\Xi(\cdot) = \sum_{k\geq 1}\pi_k \delta_{\phi_k}(\cdot)
\end{equation}
is called the PDP with base measure $G$ and
parameters $\alpha$ and $\theta$.
The base measure $G$ is non-atomic, this is used to give different names 
to the species in the process $\Xi(\cdot)$, but the unique fact that 
matters is that the species are different, the exact names are not 
important, and this explains why we ignore $G$ and one simply notes 
$PDP(\alpha,\theta)$.

\medskip

\noindent The case $\alpha=0$ is called Dirichlet process and it can be 
constructed
as an infinite extension of a Dirichlet distribution.
Examples on how PDP help to model
different phenomena can be seen
in \cite{orbanz2014lecture} and \cite{newman2005power}.

\medskip

\noindent Samples from a PDP are obtained 
from (\ref{eqn:PDP-Theta}) in the following way.
For a random measure $\Xi(\cdot)$ one takes an i.i.d. 
sequence of variables $(X_n: n\ge 1)$ with values in ${\cal X}$. 
Let $\mathbf{X}_\ell=(X_1,\dots,X_\ell)$ be a sample of size $\ell$ 
collected in a sequential way. By $K_\ell$ we note the total 
number of different species of the sample 
which are noted by $X_1^*,\dots, X_{K_\ell}^*$. For 
$j=1,\dots,K_\ell$ we note by $N^\ell_j$ the number 
of times that the species $X_j^*$ is observed in the sample, 
so $\ell=\sum_{j=1}^{K_\ell} N^\ell_j$.
Further we do not take into account the order 
of the species in the sample, if needed one can enumerate their 
frequencies in their decreasing order. 
So, $(N^\ell_j: j=1,\dots, K_\ell)$ 
means the multiset of frequencies (that is a set where the values can be 
repeated). 

\medskip

\noindent The conditional probability for a new observation
$X_{\ell+1}$ is, see \cite{buntine2012bayesian},
\begin{equation}
\label{condtr}
\PP(X_{\ell+1}=\bullet \, | \,  \mathbf{X}_\ell) = \frac{\theta + 
\alpha K_\ell}
{\theta + \ell} G(\cdot) + \sum_{j=1}^{K_\ell} 
\frac{N^\ell_j - \alpha}{\theta + \ell} \delta_{X_j^*}\;.
\end{equation}
So, the observation $X_{\ell+1}$ is part of the 
species $X^*_j$ already observed with probability
$\frac{N^\ell_j -\alpha}{\theta + \ell}$, and $X_{\ell+1}$ defines a new
species with probability $\frac{\theta + \alpha K_\ell}{\theta + \ell}$.
In this last case the new species $X_{\ell+1}=X_{K_\ell+1}^*$ is distributed
as $G$ independently of the species already discovered, and
$\ell+1$ is said to be the discovery time of a new species.
That is, 
the transition probability (\ref{condtr}) states the probability of discovering
a new species and gives a different name to it, the important point is that it 
is different to the previous ones.

\medskip

\section{Bayesian entropy}
\label{sec:entropia}

\noindent To define the Bayesian entropy one assumes a prior distribution 
and makes the estimation of entropy based
upon the posterior distribution given the sample.
We will introduce Bayesian entropy in the context
of PDP following closely, as mentioned in the introduction, the results 
in \cite{archer2014bayesian}, and also
\cite{archer2012bayesian} and \cite{chao2003nonparametric}.
To do so, we need to recall the definition of entropy.
\noindent Let $\pi$ be a distribution, the Shannon entropy is defined as
\begin{equation*}
    H(\pi) = - \sum_{i=1}^\infty \pi_i \log(\pi_i).
\end{equation*}

\noindent For further computations it is useful to introduce 
the digamma function and some of its properties, which can be found in
\cite{abramowitz1972handbook} and \cite{alzer1997some}.
This function is the logarithmic derivative of the Gamma function:
\begin{equation*}
\psi(x) = \frac{d}{dx}\log(\Gamma(x))
=\frac{\Gamma'(x)}{\Gamma(x)},
\end{equation*}
where $\Gamma(x)=\int_0^\infty t^{x-1}e^{-t}dt$.
From $\Gamma(x+1)=x\Gamma(x)$, one gets
$\psi(x+1) = \psi(x) + 1/x$ for $x>0$, that implies
\begin{equation}
\label{recmult}
x\psi(x+1)-(x-1)\psi(x)=\psi(x)+1, \; x>0.
\end{equation}
The digamma function is increasing for $x>0$ and then 
$x\psi(x+1)-(x-1)\psi(x)$ is also increasing for $x>0$.
Since $\psi(2) > 0$,
then $x\psi(x+1)>(x-1)\psi(x)$ when $x\ge 1$.  
The digamma function admits the following bounds in terms 
of the logarithmic function, see \cite{alzer1997some}: 
\begin{equation}
\label{eqn:psi-cotas} 
\log(x) - \frac{1}{x} \leq \psi(x) \leq \log(x)
-\frac{1}{2x}, \quad x>0.
\end{equation}
\noindent For $x$ sufficiently big the digamma function can be 
approximated by
\begin{equation}
\label{eqn:psi-approx}
\psi(x)=\log(x)-\frac{1}{2x}+o\left(\frac{1}{x}\right).
\end{equation}

\subsection{Entropy for the Poisson-Dirichlet Process}
\label{subsec:PDP-entropy}

\noindent Let $\mathbf{X}_\ell=(X_1,\dots,X_\ell)$ be a sample 
following a distribution $\pi$. 
The Bayesian approach for estimating the entropy requires to assume a 
prior distribution $\pi$ and estimate the posterior distribution. The 
least square Bayes estimator has the shape: $\EE(H(\pi) | \mathbf{X}_\ell)$.

\medskip

\noindent When one takes a PDP as prior, the sample $\mathbf{X}_\ell$ should 
be obtained from the random measure $\Xi$, given by \eqref{eqn:PDP-Theta}.
But, as we mentioned before, we can omit any reference to $G$, so the sample 
is obtained from the weight distribution $\pi$
and we will refer
to the process and its weight distribution 
indistinctly by the same symbol, that is, the prior is 
$\pi\sim PDP(\alpha, \theta)$. In \cite{archer2014bayesian} the 
prior mean of $H(\pi)$ 
is proven to be, 
\begin{equation*}
\E(H(\pi)) = \psi(\theta+1) - \psi(1-\alpha).
\end{equation*}

\noindent We are interested in finding the posterior mean of $H(\pi)$, 
after seeing a sample. To describe the posterior distribution 
consider the sample $\mathbf{X}_\ell$ with 
$K_\ell$ different species and 
frequencies $N^\ell_1,\dots,N^\ell_{K_\ell}$.
To simplify notation put $K_\ell=k$ and $N^\ell_j=n_j$ for $j=1,\dots,k$.
In \cite{ishwaran2003} it was shown that the posterior 
distribution 
$\pi_{post} = (p_1,\dots,p_k,(1-\sum_{j=1}^k p_j)\pi')$ is
given by the mixture
\begin{eqnarray*}
(p_1,\dots,p_k,1-\sum_{j=1}^k p_j) 
&\sim& \text{\normalfont Dirichlet}(n_1-\alpha, \dots, 
n_k-\alpha,\theta + \alpha k) \\ 
\pi'=(\pi_1',\pi_2',\dots) &\sim& PDP(\alpha,\theta +\alpha k).
\end{eqnarray*}

\noindent Hence, the probability of belonging to some species $X_j^*$
already present in the sample is $p_j$ for $j=1,\dots,k$; 
and the probability to belong to a new species  
is $1-\sum_{j=1}^k p_j$, where the distribution of 
these probabilities depend on the frequencies $(n_j)$ and $k$.
In the event that a new species is discovered it will
be part of a specific species $i$ with weight $\pi'_i$.

\medskip

\noindent The species $X_i^*$ related to the prior distribution $\pi$, 
is not the 
same as the species $X_i^*$ in the posterior distribution $\pi_{post}$,
because the index taken after observing the sample is arbitrary. But, this
index discrepancy does not cause any problem since 
the ordering of $\pi_i$ is not important in $H(\pi)$ and the transition probability 
for the discovery of a new species and for the species that have been discovered 
in the past continues to have the weights given by (\ref{condtr}).
Also, the posterior distribution of $\pi$ is 
represented by a realization $\pi_{post}$ whose ordering
is totally different from the ordering of $\pi$, 
this realization is only one representation of the posterior distribution. 

\medskip

\noindent The Bayes estimator of the posterior mean of the entropy 
under the PDP prior, at step $\ell$, will be defined as
\begin{equation*}
    \widehat{H}^\ell_{PDP} = \E(H(\pi) | \mathbf{X}_\ell).
\end{equation*}
We will write $H$ instead of $H(\pi)$ when there is no confusion, so 
$\widehat{H}^\ell_{PDP} = \E(H | \mathbf{X}_\ell)$.
In \cite{archer2014bayesian} it was shown 
that the posterior mean of $H$ under the PDP prior is,
\begin{equation}
\label{eqn:entropy-post-PDP}
\widehat{H}^\ell_{PDP} = \psi(\theta+\ell+1) -
\frac{\theta +\alpha k}{\theta+\ell}\psi(1-\alpha) 
- \frac{1}{\theta+\ell}\sum_{i=1}^k   
(n_i-\alpha)\psi(n_i-\alpha+1).
\end{equation}

\medskip

\noindent Let ${\widehat{\pi}}^\ell$ be the vector of empirical probabilities 
${\widehat{\pi}}^\ell_i=n_i/\ell$, for $i=1,\dots,k$, and ${\widehat{\pi}}^\ell_i=0$ 
for $i>k$,  given  by the sample
$\mathbf{X}_\ell$. The Maximum Likelihood 
Estimator (MLE) of the entropy, at step $\ell$, under multinomial likelihood,
is given by
\begin{equation}
\label{eqn:entropy-MLE}
\widehat{H}^\ell_{MLE} = H(\widehat{\pi}^\ell) =
- \sum_{i=1}^\infty \widehat{\pi}^\ell_i \log({\widehat{\pi}^\ell_i}),
\end{equation}
which is a biased estimator.  In \cite{archer2014bayesian} it is shown 
that when $K_\ell/\ell$ converges
in probability to $0$, then $\widehat{H}^\ell_{PDP}$ satisfies the following
consistency property,
\begin{equation}
\label{converg}
|\widehat{H}^\ell_{PDP} - \widehat{H}^\ell_{MLE}| \to 0 \hbox{ as }
\ell\to \infty.
\end{equation}

\subsection{Bounds for the posterior PDP entropy}
\label{secbo}

\noindent Let us obtain lower and upper bounds for the entropy when the
sample size is fixed. This is made firstly when the number of species 
is fixed and after over all possible number of species in the sample.

\medskip

\begin{proposition}
\label{prop:entropy-post-PDP-desigualdad}
For a sample $\mathbf{X}_\ell$ of a PDP$(\alpha,\theta)$, with $k$ 
different species the entropy is upper and lower bounded by,
\begin{eqnarray*}
\E(H|\mathbf{X}_\ell)&\le& \psi(\theta\!+\!\ell\!+\!1) -
\frac{\theta \!+\!\alpha k}{\theta\!+\!\ell}\psi(1\!-\!\alpha)
\!-\! \frac{1}{\theta\!+\!\ell}\sum_{i=1}^k
({\on}_i\!-\!\alpha)\psi({\on}_i\!-\!\alpha\!+\!1);\\
\E(H|\mathbf{X}_\ell)&\ge&
 \psi(\theta\!+\!\ell\!+\!1) 
-\frac{\theta\!+\!\alpha k}{\theta\!+\!\ell}\psi(1\!-\!\alpha)
\!-\! \frac{1}{\theta\!+\!\ell}\sum_{i=1}^k
({\un}_i\!-\!\alpha)\psi({\un}_i\!-\!\alpha\!+\!1);
\end{eqnarray*}
where the vectors of frequencies $({\on}_i:i=1,\dots,k)$ 
and $({\un}_i:i=1,\dots,k)$ of the maximal
entropy and the minimal entropy respectively, 
have the following structures up to index permutation: 
$$
{\on}_i=\lfloor \ell/k \rfloor , i=1,\dots, l_k, 
\quad {\on}_i=\lfloor \ell/k \rfloor +1, i=l_k+1,\dots, l_k+h_k
$$
where $\lfloor x \rfloor$ is the biggest integer smallest or equal 
to $x$, $h_k=\ell-k\lfloor \ell/k \rfloor$ and $l_k=k-h_k$; and
$$
{\un}_k=\ell-(k-1) \hbox{ and } {\un}_i=1, \; i=1,\dots,k-1.  
$$
Moreover, when one looks for the global bounds on all entropy maxima for 
$k\in \{1,\dots,\ell\}$,  
one finds that: the global maximum is attained when
the $\ell$ elements of the sample belong to different species and the 
global minimum is attained when          
the $\ell$ elements of the sample belong to a unique species. This is, 
$$
\min_{\mathbf{Y}_\ell}\E(H|\mathbf{Y}_\ell) \le \E(H|\mathbf{X}_\ell)
\le \max\limits_{\mathbf{Y}_\ell}\E(H|\mathbf{Y}_\ell)
$$
with
\begin{eqnarray}
\label{supt}
\max\limits_{\mathbf{Y}_\ell}\E(H|\mathbf{Y}_\ell)\!\!&=& \!\! 
\psi(\theta\!+\!\ell\!+\!1)
- \psi(1-\alpha)-\frac{\ell}{\theta+\ell},\\ 
\label{inft}
\min\limits_{\mathbf{Y}_\ell}\E(H|\mathbf{Y}_\ell)\!\!
&=& \!\! \psi(\theta\!+\!\ell\!+\!1) \!-\! 
\frac{(\theta\!+\!\alpha)\psi(1\!-\!\alpha)}{\theta\!+\!\ell}
\!-\! \frac{(\ell\!-\!\alpha)\psi(\ell\!-\!\alpha\!+\!1)} 
{\theta\!+\!\ell}.
\end{eqnarray}
\end{proposition}

\begin{proof}
\noindent We will take into account that $-\psi(1-\alpha)>0$.
Let us first prove the extremal entropies for a fixed $k$. 
If $k=1$ there nothing to examine because $n_1=\ell$ and one simply 
computes the entropy. 

\medskip

\noindent Let $k>1$. Take two species $i\neq j$ and set $n_i=n$, $n_j=m$. 
Assume $n>1$. We will fix when
the entropy grows when one makes the change $n\to n-1$, 
$m\to m+1$ and all other frequencies $n_l$ are equal,
so the number of classes continues to be $k$ and the sum
of their frequencies continues to be $\ell$.
This change makes the entropy grow if and only if the following
inequality holds (we take into account that there 
is a minus in front of the third term at the right hand side in 
(\ref{eqn:entropy-post-PDP})), 
\begin{eqnarray*}
&{}&(n-1-\alpha)\psi(n-\alpha)+(m+1-\alpha)\psi(m+2-\alpha)\\
&\le& 
(n-\alpha)\psi(n-\alpha+1)+(m-\alpha)\psi(m-\alpha+1).
\end{eqnarray*}
From (\ref{recmult}) this is equivalent to
$$
0\le -\psi(m-\alpha+1)-1+\psi(n-\alpha)+1=\psi(n-\alpha)-\psi(m-\alpha+1).
$$
But this is equivalent to $m+1\le n$. So, when this last inequality holds 
we make the change $n\to n-1$ and $m\to m+1$. (Note that if $n=m+1$ the 
change leaves the set of frequencies invariant because the new pair is the 
same, $m$, $m+1$). Therefore the maximal entropy for $k$ classes is 
attained by the following structure of frequencies:
$$
n_i=\lfloor \ell/k \rfloor , i=1,\dots, l_k, 
\quad n_i=\lfloor \ell/k \rfloor +1, i=l_k+1,\dots l_k+h_k  
$$
with $h_k=\ell-k\lfloor \ell/k \rfloor$ and $l_k=k-h_k$. This 
is the frequencies are 'as equal as possible'.  

\medskip

\noindent On the opposite when $m+1\ge n$, the change $n\to n-1$, 
$m\to m+1$, makes the entropy decrease. So, the minimal entropy 
structure of frequencies is given by $n_1=\ell-(k-1)$ and the rest of 
$k-1$ species have frequency $1$. Therefore the first two 
inequalities of the Proposition are shown.

\medskip

\noindent Now  for obtaining the global maxima and 
minima we must see what happens with the extreme
solutions for different $k$'s. 

\medskip
  
\noindent This is based upon the following observation. Assume we have
$k<\ell$ number of species with frequencies $(n_1,\cdots,n_k)$ and 
$n_k>1$. Let us see what happens when we change this structure of frequencies
to one that contains $k+1$ species and  $(n_1,\cdots,n_{k-1},n_k-1,1)$, 
so with $n_{k+1}=1$. We 
claim that this operation makes the entropy strictly bigger. In fact
by (\ref{eqn:entropy-post-PDP}) the claim is equivalent to
$$
-\alpha\psi(1-\alpha)-(n_k-1-\alpha)\psi(n_k-\alpha)-(1-\alpha)\psi(2-\alpha)
> -(n_k-\alpha)\psi(n_k+1-\alpha).
$$
By using (\ref{recmult}) this last inequality is equivalent to
\begin{equation}
\label{agrandar}
-\alpha\psi(1-\alpha)-(1-\alpha)\psi(2-\alpha)+\psi(n_k-\alpha)+1>0.
\end{equation}
Since $\psi(n_k-\alpha)\ge \psi(2-\alpha)$ it suffices to check the 
inequality (\ref{agrandar}) for $n_k=2$. When in the expression at 
the left hand side in (\ref{agrandar}) we set $n_k=2$ we get, 
$$
\alpha(\psi(2-\alpha)-\psi(1-\alpha))+1,
$$
which is strictly positive, so (\ref{agrandar}) holds and the claim 
is satisfied.

\medskip

\noindent Then, if one takes the maximal configuration for $k<\ell$ species,
we know that there exists a frequency, that we can assume is the $k-$th
one, that satisfies $n_k>1$. So, by making the above operation gives a 
configuration of frequencies of a total number of species $k+1$ and such that 
the entropy increases strictly. In particular the maximal entropy for 
$k+1$ species is strictly bigger than the maximal entropy for $k$ species.
Then, (\ref{supt}) is proven.

\medskip

\noindent Finally when we make the above operation from the 
minimal configuration of 
$k$ species we retrieve the minimal configuration of the $k+1$ 
species and so the minimal entropy for $k$ species is strictly lower than 
the minimal entropy for $k+1$ species. So, 
(\ref{inft}) follows. The result is shown.  
\end{proof}

\medskip

\begin{remark}
\label{infini}
From (\ref{inft}) and since $-\psi(1-\alpha)>0$, we get 
$$
\min\limits_{\mathbf{Y}_\ell}((\theta\!+\!\ell)\E(H|\mathbf{Y}_\ell))\!
\ge \!(\theta\!+\!\ell) \psi(\theta\!+\!\ell\!+\!1) 
\!-\!(\ell\!-\!\alpha)\psi(\ell\!-\!\alpha\!+\!1),
$$
where $\theta>-\alpha$. On the other hand 
for every real $h>0$ we have $(x+h)\log(x+h+1)-x\log (x+1) 
\to \infty$ as $x\to \infty$.
Then, by also using (\ref{eqn:psi-approx}) we get that 
$\min\limits_{\mathbf{Y}_\ell}((\theta\!+\!\ell)\E(H|\mathbf{Y}_\ell))
\to \infty$ as $\ell\to \infty$. $\Box$
\end{remark}

\medskip

\noindent The relation (\ref{converg}) shows a key property between 
the frequentist estimator based on empirical probabilities
and the Bayesian estimator based on the posterior mean under the PDP 
prior, when $\ell\to \infty$.
In next section we will study the variation of weighted estimators
when making a finite step $\ell$ to $\ell+1$, showing a property 
that is similar for both, the frequentist and the PDP cases.

\medskip

\section{One step variation of entropy and discovery of a new species}
\label{sec:variacion}

\noindent We will state and prove our main result: an equality proving  
that a weighted variation between two successive steps of the posterior 
Bayesian entropy, is nonnegative and only vanishes in the
discovery times of a new species. 
This is done in Section \ref{subsec:var-entropy-PDP}.

\medskip

\noindent Related to this result, we previously study the 
variation of the entropy when
one only computes frequencies, and how it characterizes 
discovery time of species. 

\subsection{One step variation of entropy for frequencies}
\label{subsec:var-entropy-emp}

\noindent The framework is the following one: we collect a series of 
elements that are being classified in some 
class or species, at the moment when they are observed. 
At step $\ell$ one has collected in a sequential 
way $\ell$ elements $(X_1,\dots,X_\ell)$
that are grouped into a set 
of disjoint equivalence classes which are  
enumerated in a sequential way as it first element is discovered. Let
$k_\ell$ be the number of classes at step $\ell$ 
and $(n^\ell_j: j=1,\dots,k_\ell)$ be the number of
elements in these classes, so $\ell=\sum_{j=1}^{k_\ell}n_j^\ell$.

\medskip

\noindent When a new element $X_{\ell+1}$ is observed, there are two 
possibilities: this element is in a class of an element collected before 
or at $\ell$, in this case $k_{\ell+1}=k_\ell$ and if $X_{\ell+1}$ 
belongs to the class $j$ then $n_j^{\ell+1}=n_j^{\ell}+1$. When 
$X_{\ell+1}$ is in none of the classes of the previous elements 
then a new class is discovered, so $k_{\ell+1}=k_\ell+1$, 
$n_{k_\ell+1}^{\ell+1}=1$ at 
step $\ell+1$ and the frequencies of 
the classes that do not contain $X_{\ell+1}$ remain unchanged 
from $\ell$ to $\ell+1$. The entropy at step $\ell$ is
\begin{equation*}
H_\ell = - \sum_{j=1}^{k_\ell}
\frac{n_j^\ell}{\ell} \log \left(\frac{n_j^\ell}{\ell}\right).
\end{equation*}
This relation is entirely similar to
(\ref{eqn:entropy-MLE}).
We set $0\log0=0$, so one can add an empty class without changing the 
entropy.  

\begin{remark}
\label{notnot1}
In general the sequence $(H_\ell: \ell\ge 1)$ is neither 
increasing nor decreasing. For instance if the observations $X_i$, 
$i=1,\dots, 4$ are such that the pairs $\{X_1,X_3\}$ and $\{X_2,X_4\}$ 
belong to the same class, but 
the classes are different, it holds $\log 2=H_2=H_4>H_3$.  $\, \Box$
\end{remark}

\noindent One has $H_\ell\le \log \ell:=H^{\hbox{\footnotesize{max}}}_\ell$,
and the equality is attained only when $k_\ell=\ell$, that is when each of 
the $\ell$ elements defines its own class. We also have $H_\ell\ge 0$ and it 
vanishes only when there is a unique class containing the $\ell$ elements.
In all the other cases both
inequalities, the upper and lower bounds, are strict. 
Also notice that $H_1=0$.  

\medskip

\noindent Below we will consider the steps $\ell$ and $\ell+1$ of 
the sequence $(H_\ell: \ell\geq 1)$. We will
note by $j^{\ell+1}\in \{1,\dots,k_{\ell+1}\}$ the index of class 
that contains observation $X_{\ell+1}$. Then, 
$n_{j^{\ell+1}}^{\ell+1}$ is the 
frequency of class $X^*_{j^{\ell+1}}=X_{\ell+1}$ at step $\ell+1$.

\begin{proposition}
\label{freqc}
The functional given by 
\begin{equation*}   
{\aL}^f_\ell=\ell(\log \ell-H_\ell), \hbox{ for } \ell\ge 1 \hbox{ and } 
{\aL}^f_0=0,
\end{equation*}  
is a nondecreasing and nonnegative functional along the trajectory
$(X_\ell: \ell\ge 1)$ and it remains constant, 
${\aL}^f_{\ell+1}={\aL}^f_\ell$,
only when a new species is discovered at $\ell+1$.
More precisely,
$\Delta^f_{\ell+1}={\aL}^f_{\ell+1}-{\aL}^f_\ell$
satisfies
\begin{equation}
\label{eqn:entropy-delta-Servet} 
\forall \ell\geq 1, \quad \Delta^f_{\ell+1}
=n_{j^{\ell+1}} \log(n_{j^{\ell+1}})
-(n_{j^{\ell+1}}\!-\!1)\log(n_{j^{\ell+1}}\!-\!1) \geq 0,
\end{equation}
and $\Delta^f_{\ell+1}=0$ only when a new class is discovered at $\ell+1$, that is
\begin{equation}
\label{eqn:delta-0-Servet}
\Delta^f_{\ell+1}=0 \Leftrightarrow n_{j^{\ell+1}}=1.
\end{equation}
Moreover,
\begin{eqnarray}
\label{nlogn1}
&{}&(\ell+1)H_{\ell+1}-\ell H_\ell\\
\nonumber
&{}&\, =\!
(\ell\!+\!1)\log (\ell\!+\!1)\!-\!\ell \log \ell
\!-\!\left(n_{j^{\ell+1}} \log(n_{j^{\ell+1}})
\!-\!(n_{j^{\ell+1}}\!-\!1)\log(n_{j^{\ell+1}}\!-\!1)\right)\! \ge \!0,
\end{eqnarray}
and vanishes only when $K_{\ell+1}=1$.
\end{proposition}

\begin{proof}
\noindent We will show (\ref{nlogn1}) at the end of the proof. 
All the other properties will follow 
when we show that
$\Delta^f_{\ell+1}$ satisfies the equality in 
(\ref{eqn:entropy-delta-Servet}). 
In fact, the inequality $\Delta^f_{\ell+1}\ge 0$ is a direct
consequence of it because $j \log j-(j-1)\log(j-1)\ge 0$. 
This implies that the functional ${\aL}^f_\ell$ is nondecreasing. 
Also we have that $j \log j-(j-1)\log(j-1)$ vanishes only if $j=1$, 
and so (\ref{eqn:delta-0-Servet}) is obtained and this ensures 
that the functional ${\aL}$ remains constant only at times 
when a new  class is discovered. 

\medskip

\noindent Notice that $\Delta^f_1={\aL}^f_1-{\aL}^f_0=0$ is 
consistent with the fact that at step $1$ a new class is discovered.  

\medskip

\noindent Let us show the equality in (\ref{eqn:entropy-delta-Servet}). 
To simplify notation, we note $j^*=j^{\ell+1}$ 
the class containing $X_{\ell+1}$ at step $\ell+1$.
Also we write $\sum\limits_{j\neq j^*}$
to mean $\sum\limits_{1\le j\le k_{\ell+1}, j\neq j^*}$. 
In the rest of the proof we note $n_j=n^{\ell+1}_j$
for $j=1,\dots,k_{\ell+1}$, so $n_{j^*}$ is the cardinality of the 
class $X^*_{j^*}$. If at step $\ell+1$ 
one has $j\neq j^*$ then the number of elements of 
the class $j$ is equal at steps $\ell$ and $\ell+1$. We have
$$
(\ell+1)H_{\ell+1} = - \sum_{j=1}^{k_{\ell+1}} n_j \log n_j 
+(\ell + 1) \log( \ell+1)
$$
and then
\begin{equation*}
(\ell\!+\!1)(\log (\ell\!+\!1)\!-\!H_{\ell\!+\!1})\!
=\!\sum_{j=1}^{k_{\ell+1}} n_j \log n_j 
\!=\!\sum\limits_{j\neq j^*} n_j \log n_j \!+\! n_{j^*}
\log n_{j^*}.
\end{equation*}
Now, the frequency of class $j^*$ at step $\ell$ is $n_{j^*}-1$, 
so in a similar way as we did for the term $\ell+1$ we get
$$
\ell(\log \ell- H_\ell)=\sum\limits_{j\neq j^*} n_j 
\log n_j +(n_{j^*}-1)
\log (n_{j^*}-1).
$$
Then, $\Delta^f_{\ell+1}=(\ell+1)(\log (\ell+1)-H_{\ell+1})
-\ell(\log \ell- H_\ell)$  
satisfies the equality in (\ref{eqn:entropy-delta-Servet}).

\medskip

\noindent Finally the equality in (\ref{nlogn1})
is directly obtained from the equality in (\ref{eqn:entropy-delta-Servet}).   
The inequality $\ge 0$ in this relation is a consequence of the increasing
property of the function $(n+1)\log(n+1)-n\log n$ for $n\ge 1$, which follows
from $(1+1/n)^n<(1+1/(n+1))^{n+1}$ for all $n\ge 1$ (and $0\log 0=0$).
\end{proof}

\medskip
    
\noindent Consider the function  
$\kappa(\ell+1)=(\ell+1)\log (\ell+1)-\ell\log \ell$ for $\ell\ge 1$. 
From $x-x^2/2\le \log (1+x)\le x$ for $x\geq 0$, we get 
$$
\frac{1}{2\ell}-\frac{1}{2\ell^2} \le \kappa(\ell+1)-(\log \ell +1) \le 
\frac{1}{\ell}\,,
$$
and for large $\ell$ we have
$\kappa(\ell+1)\approx \log \ell+1 + o(1)$.
These bounds and approximation can be applied for 
$\Delta^f_{\ell+1}=\kappa(n_{j^{\ell+1}})$.

\subsection{One step variation of the Bayesian entropy}
\label{subsec:var-entropy-PDP}

\noindent Let us consider the one step variation of Bayesian 
entropy for the PDP. Consider an i.i.d. sequence 
$(X_n: n\ge 1)$ of elements in 
${\cal X}$ chosen with a random measure $\Xi(\cdot)$ of a 
PDP$(\alpha,\theta)$ which fixes the family of 
finite samples $\mathbf{X}_\ell=(X_1,\dots,X_\ell)$, $\ell\ge 1$. 

\begin{remark}
\label{notnot2}
We note that the sequence of entropies $(\E(H|\mathbf{X}_\ell): \ell\ge 1)$ 
is neither increasing nor decreasing.
We can illustrate it with the same example used in
Remark \ref{notnot1}. So, assume the observations 
$X_i$, $i=1,\dots, 4$ are such that the pairs $\{X_1,X_3\}$ and 
$\{X_2,X_4\}$ are in the same class, 
but the classes are different. It can be checked that when 
$0\le \alpha<1/2$ and $-\alpha<\theta<1-3 \alpha$, 
it holds $\EE(H|\mathbf{X}_2)>\E(H|\mathbf{X}_3)$ and 
$\EE(H|\mathbf{X}_4)>\EE(H|\mathbf{X}_3)$.  $\, \Box$
\end{remark}
 
\noindent In the next result we will compute the one step variation 
of the posterior entropy of a PDP$(\alpha,\theta)$, 
when taking the sample 
$\mathbf{X}_{\ell+1}= (\mathbf{X}_\ell,X_{\ell+1})$.
We recall relation (\ref{supt}) that gives the maximum entropy 
for samples of size $\ell$, it is
$$
\max_{\mathbf{Y}_\ell}\E(H|\mathbf{Y}_\ell)= \psi(\theta+\ell+1)
- \psi(1-\alpha)-\frac{\ell}{\theta+\ell}.
$$
From (\ref{recmult}) we get
$$
(\theta+\ell+1)\psi(\theta+\ell+2)-(\theta+\ell)\psi(\theta+\ell+1)
=\psi(\theta+\ell+1)+1,
$$
and so,
\begin{equation}
\label{difmaxN}
(\theta\!+\!\ell\!+\!1)\max_{\mathbf{Y}_{\ell+1}}
\E(H|\mathbf{Y}_{\ell+1})
\!-\!(\theta\!+\!\ell)\max_{\mathbf{Y}_\ell}\E(H|\mathbf{Y}_\ell)
\!=\!\psi(\theta\!+\!\ell\!+\!1)\!-\!\psi(1\!-\!\alpha).
\end{equation}

\noindent Now we state our main result, satisfied by the 
functional given in (\ref{eq111}). As in the 
frequentist 
case we note by $j^{\ell+1}$ the index of the
species $X_{\ell+1}$, that is such that $X_{\ell+1}=X^*_{j^{\ell+1}}$. 

\begin{theorem}
\label{teo:Delta-ppal}
Let $(X_n: n\ge 1)$ be an i.i.d. sequence of a PDP$(\alpha,\theta)$.
The functional $({\aL}_\ell: \ell\ge 0)$ given by ${\aL}_0=0$ and
\begin{equation*}
{\aL}_\ell=(\theta+\ell)
\left(\max\limits_{\mathbf{Y}_\ell}\E(H|\mathbf{Y}_\ell)
-\E(H|\mathbf{X}_\ell)\right) 
\hbox{ for } \ell\ge 1;
\end{equation*}
is a nondecreasing and nonnegative functional along the trajectory
$(X_\ell: \ell\ge 1)$ and it
remains constant, ${\aL}_{\ell+1}={\aL}_\ell$,
only when a new species is discovered at $\ell+1$. More precisely, let
$$
\Delta_{\ell+1}={\aL}_{\ell+1}-{\aL}_\ell,
$$
and note $j^*=j^{\ell+1}$ be the index of the species $X_{\ell+1}$ and 
$n_{j^*}=n^{\ell+1}_{j^*}$ be the frequency of
this species at step $\ell+1$. Then,
\begin{equation}   
\label{nuevoDelta}
\Delta_{\ell+1} = \psi(n_{j^*}-\alpha)-\psi(1-\alpha)\ge 0
\end{equation}
and it vanishes only when $n_{j^*}=1$, that is when a new species 
is discovered at $\ell+1$. Moreover
\begin{equation}
\label{nlogn2}
(\theta+\ell+1)\E(H|\mathbf{X}_{\ell+1})-(\theta+\ell)\E(H|\mathbf{X}_\ell)
=\psi(\theta+\ell+1)-\psi(n_{j^*}-\alpha)>0.
\end{equation}
\end{theorem}

\begin{proof}
The relation (\ref{nlogn2}) will be shown at the end of the proof. 
Note that for the rest of the relations it suffices to show 
(\ref{nuevoDelta}) because $n_{j^*}\ge 1$ 
and $\psi$ is strictly increasing then the expression at the 
right hand side of (\ref{nuevoDelta}) increases strictly with $n_{j^*}$
and it vanishes only when $n_{j^*}=1$. So, let us  
show equality (\ref{nuevoDelta}).

\medskip

\noindent The sequence of mean posterior entropies is noted by
$\widehat{H}_\ell=\E(H|\mathbf{X}_\ell)$, $\ell\ge 1$.
From (\ref{eqn:entropy-post-PDP}) we have
\begin{equation*}
(\theta\!+\!\ell)\widehat{H}_\ell =(\theta\!+\!\ell) 
\psi(\theta\!+\!\ell\!+\!1) - 
(\theta\!+\!\alpha k_\ell)\psi(1\!-\!\alpha) - \sum_{i=1}^{k_\ell}
(n^\ell_i\!-\!\alpha)\psi(n^\ell_i\!-\!\alpha+1).
\end{equation*}

\noindent Let us define,
\begin{equation}
\label{defet}
\eta_{\ell+1}= (\theta+\ell+1)\widehat{H}_{\ell+1}-
(\theta+\ell)\widehat{H}_\ell.
\end{equation}
From the definitions of $\Delta$ and $\eta$ and equality (\ref{difmaxN}) 
we get
\begin{equation*}
\Delta_{\ell+1}=\psi(\theta+\ell+1)-\psi(1-\alpha)-\eta_{\ell+1}.
\end{equation*}
So, instead of proving results for ${\aL}_\ell$ and $\Delta_{\ell}$ we 
will do it for $\eta_{\ell}$.

\medskip

\noindent Let $K_{\ell+1}=k_{\ell+1}$. We note by $n_j=n^{\ell+1}_j$ 
the frequency of class $X^*_j$ for $j=1,\dots, k_{\ell+1}$.
We will show that the following relation holds for $\ell\geq 1$:
\begin{equation}
\label{eqn:Delta-j}
\eta_{\ell+1} = \psi(\theta\!+\! \ell \!+\!1)-\psi(n_{j^*}\!-\!\alpha).
\end{equation}
Since this implies (\ref{nuevoDelta}), the result of the Theorem will be 
satisfied.

\medskip

\noindent We first show the case $k_{\ell+1}=k_\ell+1$, 
so $j^*=k_{\ell+1}$
is the index of a new class and $n_{j^*}=n_{k_{\ell+1}}=1$.
The mean posterior entropy 
$\widehat{H}_{\ell+1}$ is computed from (\ref{eqn:entropy-post-PDP}) but
with the sample size $\ell+1$, the number of 
species $k_{\ell+1}=k_\ell+1$, the frequencies $n_j$ are unchanged 
for $j=1,\dots,k_\ell$ and the 
frequency for the new species is $n_{k_\ell+1}=1$. Then,
\begin{eqnarray*}
(\theta+\ell+1)\widehat{H}_{\ell+1} &=&(\theta+\ell+1) 
\psi(\theta+\ell+2) - 
(\theta +(k_\ell+1)\alpha)\psi(1-\alpha) \\
&{}& \; - \sum_{i=1}^{k_\ell+1} (n_i-\alpha)\psi(n_i-\alpha+1).
\end{eqnarray*}
Now we use (\ref{recmult}) on $x=\theta+\ell+2$ to get 
$(\theta+\ell+1)\psi(\theta+\ell+2)=(\theta+\ell)
\psi(\theta+\ell+1)+\psi(\theta+\ell+1)+1$,
decompose the first term at the right hand side, 
separate the term $k_\ell+1$ in the sum and use $n_{k_\ell+1}=1$, to obtain,
\begin{eqnarray*}
(\theta+\ell+1)\widehat{H}_{\ell+1} &=& 
(\theta+\ell+1)\psi(\theta+\ell+1) + 1
-(\theta +(k_\ell+1)\alpha)\psi(1-\alpha) \\
&{}& \; 
-\sum_{i=1}^{k_\ell} (n_i-\alpha)\psi(n_i-\alpha+1)
-(1-\alpha)\psi(2-\alpha).
\end{eqnarray*}
On the other hand,
\begin{eqnarray*}
(\theta+\ell)\widehat{H}_{\ell} &=& (\theta+\ell)\psi(\theta+\ell+1) 
-(\theta +\alpha k_\ell)\psi(1-\alpha) \\
&{}& \; -\sum_{i=1}^{k_\ell} (n_i-\alpha)\psi(n_i-\alpha+1).
\end{eqnarray*}
By using 
$(1-\alpha)\psi(2-\alpha)=(1-\alpha)\psi(1-\alpha)+1$, we get
$$
\eta_{\ell+1}= (\theta+\ell+1)\widehat{H}_{\ell+1}
-(\theta+\ell)\widehat{H}_{\ell}
= \psi(\theta+\ell+1)-\psi(1-\alpha).
$$
So, relation (\ref{eqn:Delta-j}) is shown when $k_{\ell+1}=k_\ell+1$.

\medskip

\noindent Let us show (\ref{eqn:Delta-j}) when $k_{\ell+1}=k_\ell$. 
For $j\neq j^*$  we have $n_j=n_j^{\ell+1}=n_j^\ell$, and
for $j^*$ we have $n_{j^*}^\ell=n_{j^*}-1$. We will simplify some notation 
on sums and put $\sum_{i\neq j^*}=\sum_{i=1,\dots, k, i\neq j^*}$. From, 
\begin{eqnarray*}
(\theta\!+\!\ell\!+\!1)\widehat{H}_{\ell+1} &=& (\theta\!+\!\ell\!+\!1) 
\psi(\theta\!+\!\ell\!+\!2) - (\theta\!+\!\alpha k_\ell)
\psi(1\!-\!\alpha) \\ 
&{}&\;  -\sum_{i\neq j^*} (n_i\!-\!\alpha)\psi(n_i\!-\!\alpha\!+\!1)
-(n_{j^*}\!-\!\alpha)\psi(n_{j^*}\!\!-\!\alpha+1),
\end{eqnarray*}
and
$$
(\theta+\ell)\widehat{H}_{\ell} = (\theta\!+\!\ell) 
\psi(\theta\!+\!\ell\!+\!1) -
(\theta\!+\!\alpha k_\ell )\psi(1\!-\!\alpha)-\sum_{i=1}^{k_\ell} 
(n_i\!-\!\alpha)\psi(n_i\!-\!\alpha\!+\!1),
$$
we obtain
\begin{eqnarray*}
\eta_{\ell+1}&=&(\theta+\ell+1)\widehat{H}_{\ell+1} - 
(\theta+\ell)\widehat{H}_{\ell} \\ 
&=& (\theta+\ell+1) \psi(\theta+\ell+2) -(\theta+\ell) 
\psi(\theta+\ell+1)\\
&{}&\;\; -(n_{j^*}-\alpha)\psi(n_{j^*}-\alpha+1)
+(n_{j^*}-1-\alpha)\psi(n_{j^*}-\alpha).
\end{eqnarray*}
By using (\ref{recmult}) in $x=\theta+\ell+1$ and $x=n_{j^*}-\alpha$ we get,
\begin{eqnarray*}
&{}& (\theta+\ell+1) \psi(\theta+\ell+2) -(\theta+\ell) 
\psi(\theta+\ell+1)=\psi(\theta+\ell+1)+1 \hbox{ and }\\
&{}&
-(n_{j^*}-\alpha)\psi(n_{j^*}-\alpha+1)+ 
(n_{j^*}-\alpha-1)\psi(n_{j^*}-\alpha)
=-\psi(n_{j^*}-\alpha)-1.
\end{eqnarray*}
Therefore
$$
\eta_{\ell+1}=\psi(\theta+\ell+1)-\psi(n_{j^*}-\alpha),
$$
and the relation (\ref{eqn:Delta-j}) is shown 
for the case $k_{\ell+1}=k_\ell$.

\medskip 

\noindent To finish the proof of the Theorem
let us show (\ref{nlogn2}). It follows from definition (\ref{defet}),
the relation (\ref{eqn:Delta-j}), the inequality 
$\theta>-\alpha$ and $\psi$ is increasing.
\end{proof}

\medskip

\begin{remark}
\label{normal}
Set ${\wH}^{\hbox{\footnotesize{max}}}_\ell
=\max_{{\mathbf{Y}_\ell}}\E(H|\mathbf{Y}_\ell)$.
We have analyzed the variation,
$$
\Delta_{\ell+1}=(\theta+\ell+1)({\wH}^{\hbox{\footnotesize{max}}}_{\ell+1}
-{\wH}_{\ell+1})-(\theta+\ell)({\wH}^{\hbox{\footnotesize{max}}}_\ell
-{\wH}_\ell).
$$
Note that any other weights would produces only trivial changes or would
lead to the analysis of the variation weighted with the entropy. In fact
if one considers
$$
c_{\ell+1}=(\theta+\ell+1)(a_{\ell+1}
-{\wH}_{\ell+1})-(\theta+\ell)(a_\ell-{\wH}_\ell),
$$
then
$c_{\ell+1}=\Delta_{\ell+1}+(\theta+\ell+1)(a_{\ell+1}
-{\wH}^{\hbox{\footnotesize{max}}}_{\ell+1})
-(\theta+\ell)(a_\ell-{\wH}^{\hbox{\footnotesize{max}}}_\ell)$,   
so it suffices to add to $\Delta_{\ell+1}$ a
deterministic sequence depending on $\ell$. If one considers
$$
c'_{\ell+1}=b_{\ell+1}({\wH}^{\hbox{\footnotesize{max}}}_{\ell+1}
-{\wH}_{\ell+1})
-b_\ell({\wH}^{\hbox{\footnotesize{max}}}_\ell-{\wH}_\ell),
$$
one gets
\begin{eqnarray*}
c'_{\ell+1}&=&
b_\ell\left(\frac{b_{\ell+1}}{b_\ell}
({\wH}^{\hbox{\footnotesize{max}}}_{\ell+1}
-{\wH}_{\ell+1})-
({\wH}^{\hbox{\footnotesize{max}}}_\ell-{\wH}_\ell)\right)\\
&=&
b_\ell\left(\frac{b_{\ell+1}}{b_\ell}-\frac{\theta\!+\!\ell\!+\!1}
{\theta+\ell}\right)
({\wH}^{\hbox{\footnotesize{max}}}_{\ell+1}-{\wH}_{\ell+1}) 
+\frac{b_\ell}{\theta+\ell}\Delta_{\ell+1}.
\end{eqnarray*}
When we modify both, the additive and the
multiplicative terms, in $\Delta_{\ell+1}$ we
get a combination of above situations. $\, \Box$
\end{remark}

\medskip

\begin{remark}
\label{maxcomp}
In the frequentist case the weighted difference between maximal 
entropies at steps $\ell+1$ and $\ell$ is,
$$
d^f_{\ell+1}=(\ell+1)H^{\hbox{\footnotesize{max}}}_{\ell+1}-
\ell H^{\hbox{\footnotesize{max}}}_\ell=
(\ell+1)\log (\ell+1)-\ell\log \ell.
$$
From (\ref{difmaxN}), in the Bayesian PDP case 
the weighted difference of posterior entropies is,
$$
d_{\ell+1}=(\theta+\ell+1){\wH}^{\hbox{\footnotesize{max}}}_{\ell+1}
-(\theta+\ell){\wH}^{\hbox{\footnotesize{max}}}_{\ell}
=\Delta_{\ell+1}+\eta_{\ell+1}=\psi(\theta+\ell+1)-\psi(1-\alpha).
$$
For big $\ell$ we have that $d^f_{\ell+1}$ is of the order of 
$\log \ell +1$ while
from (\ref{eqn:psi-approx}) one gets that $d_{\ell+1}$ is of the order of
$\log \ell-\psi(1-\alpha)$ (we recall that $-\psi(1-\alpha)>0$). $\; \Box$.
\end{remark}   

\begin{remark}
\label{boundslimit}
\noindent Now, by applying the relations (\ref{eqn:psi-cotas}) and 
(\ref{eqn:psi-approx}) satisfied by the digamma function,  
from Theorem \ref{teo:Delta-ppal} we get the
following bounds for the weighted entropy variation $\eta_{\ell+1}=
(\theta+\ell+1)\widehat{H}_{\ell+1}-(\theta+\ell)\widehat{H}_\ell$ 
given by (\ref{nlogn2}),
\begin{eqnarray*}
\eta_{\ell+1} &\geq& \log(\theta\!+\!\ell\!+\!1)
-\frac{1}{\theta\!+\!\ell\!+\!1} 
-\log(n_{j^*}\!-\!\alpha) +
\frac{1}{2(n_{j^*}-\alpha)},\\
\eta_{\ell+1} &\leq&
\log(\theta\!+\!\ell\!+\!1) -\frac{1}{2(\theta\!+\!\ell\!+\!1)}
- \log(n_{j^*}\!-\!\alpha) + \frac{1}{n_{j^*}\!-\!\alpha}.
\end{eqnarray*}
When $\ell$ is sufficiently big one has,
$$
\eta_{\ell+1} \approx \log(\theta+\ell+1) 
- \frac{1}{2(\theta+\ell+1)} \hbox{ if } k_{\ell+1}=k_\ell+1;
$$
and if also $n_{j^*}$ is also sufficiently big, then
$$
\eta_{\ell+1} \approx \log(\theta\!+\!\ell\!+\!1) -
\frac{1}{2(\theta\!+\!\ell\!+\!1)} - \log(n_{j^*}\!-\!\alpha)
+ \frac{1}{2(n_{j^*}\!-\!\alpha)} \hbox{ if } k_{\ell+1}=k_\ell.
$$
\end{remark}

\medskip

\begin{remark}
\label{entdisc}
One can check that (\ref{nlogn2}) also holds for $\ell=0$,
where for the posterior mean entropy (\ref{eqn:entropy-post-PDP}),
when $\ell=0$, one takes $k=0$, and so $\theta\widehat{H}^0_{PDP} = 
\theta \psi(\theta+1) -\theta \psi(1-\alpha)$. So, 
by applying the telescopic property to (\ref{nlogn2}) we get
$$
(\theta+\ell)\widehat{H}_{\ell}=
C_\ell(\alpha, \theta)
-\sum_{i=1}^\ell \psi(n^*(i)-\alpha),
$$
where $C_\ell(\alpha, \theta)=\left(\sum_{i=1}^\ell 
\psi(\theta+i)\right)+\theta \psi(\theta+1)-\theta \psi(1-\alpha)$,
and $n^*(i)=\#\{1\leq j \leq i \, : \, X_j=X_i \}$ is 
the frequency of the class of the species $X_i$
at step $i$. Therefore, the only part of the entropy
depending on the sample is $-\sum_{i=1}^\ell
\psi(n^*(i)-\alpha)$. The terms $-\psi(n^*(i)-\alpha)$ strictly decreases 
with $n^*(i)$ (note that $-\psi(n^*(i)-\alpha)$ is positive when
$n^*(i)=1$, negative if $n^*(i)\ge 3$ and the sign of 
$-\psi(2-\alpha)$ depends on $\alpha\in [0,1)$). So, the terms 
$-\psi(n^*(i)-\alpha)$
can be seen as the `discovery value' of observing 
the species $X_i$ at step $i$, and so, up to the additive deterministic 
term, the entropy turns out to be the `discovery' values at the successive 
steps of the sample.
On the other hand, from (\ref{nuevoDelta}) we get that 
$$
\aL_\ell=\sum_{i=1}^\ell (\psi(n^*(i)-\alpha)-\psi(1-\alpha))
$$
is a sum of positive rewards for reinforcing what is already known that is
going in the opposite direction of discovery. Thus
the reward at step $i$, attains the minimum $0$ for the discovery of 
a new species. Differently to entropy, here no 
additional deterministic term 
depending on $\ell, \alpha$ and $\theta$ is required.
\end{remark} 

\medskip

\subsection{A common framework for the frequentist and the PDP cases}
The equations (\ref{nuevoDelta}) and (\ref{eqn:entropy-delta-Servet})
have the same shape, both are measuring the 
weighted differences of the distance of successive 
entropies to the maximal entropies and 
both formulae express that these differences only depend
on the updated frequency of the species of the new element. 
In fact this result holds for the class of entropies that satisfy:
\begin{equation}
\label{genent}
w(\ell){\bH}_\ell=u(a+\ell)-b-\sum_{i=1}^k (u(n^\ell_i-c)+v).
\end{equation}
Here $w(\ell)$ is a strictly positive function and increasing 
in $\ell$ and $u$ is a real 
function defined on $\NN-c=\{n-c: n\ge 1\}$ and it satisfies 
\begin{equation}
\label{constct1}
u(n+1-c)-u(n-c) \hbox{ is increasing for } n\ge 1. 
\end{equation}
The quantities $a,b,c,v$
are constants that satisfy the conditions
\begin{equation}
\label{constct}
0\le c<1, -c\le a \hbox{ and } 2u(1-c)+v< u(2-c).
\end{equation}
Notice that $H_\ell$ can be written as $\bH_\ell$ with $w(\ell)=\ell$, 
$u(x)=x\log x$ and $a=b=c=v=0$; and ${\wH}_\ell$ can be also 
written in the form $\bH_\ell$ with $w(\ell)=\theta+\ell$, 
$u(x)=x\psi(x+1)$, $a=\theta, 
b=\theta \psi(1-\alpha), c=\alpha, v=\alpha \psi(1-\alpha)$. 
In both cases $0\le c<1$. The second part in (\ref{constct}) holds 
for the PDP because $\theta>-\alpha$ and the third 
part of (\ref{constct}) holds in the frequentist case 
because it is equivalent to $2\log(1)\le \log 2$ and in the PDP case 
(\ref{constct}) becomes 
$(1-\alpha)\psi(2-\alpha)+\alpha \psi(1-\alpha)< \psi(2-\alpha)+1$ 
which is satisfied. In relation to (\ref{constct1}),
in the PDP case it follows from $x\psi(x+1)-(x-1)\psi(x)$ 
increasing in $x>0$
and in the frequentist case (\ref{constct1}) it is a consequence of 
$(n+2)\log (n+2)-(n+1)\log(n+1)>(n+1)\log (n+1)-n\log n$ for $n\ge 0$. 

\medskip
 
\noindent We will see that the conditions (\ref{constct1}) 
and (\ref{constct}) are sufficient to show that the properties proven 
for the variation of differences between maximal entropies and entropies 
for the cases $(H_\ell)$ and $(\wH_\ell)$, also hold 
for the entropy $({\bH}_\ell)$ written in (\ref{genent}).

\medskip

\noindent In order to retrieve the results in Proposition
\ref{prop:entropy-post-PDP-desigualdad} we need to analyze
what happens when, for two species $i\neq j$ with $n^\ell_i=n>1$
and $n^\ell_j=m$, one makes the change $m\to m+1$ and $n\to n-1$,
and all other frequencies $n_l$ remain equal. The entropy increases 
if and only if $u(n-c-1)+u(m-c+1)\le u(n-c)+u(m-c)$,
or equivalently $u(m-c+1)-u(m-c)\le u(n-c)-u(n-c-1)$. From
(\ref{constct1}) this holds if and only if $m+1\le n$. 

\medskip

\noindent The second requirement has to do with the following
change: for a class $i\le k$ with $n_i=n>1$ we set $n\to n-1$
and $k\to k+1$ so there is a new class with $n_{k+1}=1$.
This change makes the entropy increase if
$u(n-c-1)+u(1-c)+v< u(n-c)$ or equivalently if
$u(1-c)+v< u(n-c)-u(n-c-1)$ when $n>1$. From (\ref{constct1}) we get
that it suffices that the following inequality holds
$2u(1-c)+v< u(2-c)$, 
which is the second condition in (\ref{constct}).

\medskip

\noindent When these conditions take place the maximal entropy is attained
when all the classes are singletons, so
$$
w(\ell){\bH}^{\hbox{\footnotesize{max}}}_\ell
=u(a+\ell)-b-\sum_{i=1}^\ell (u(1-c)+v)
$$
Hence,
$$
w(\ell+1){\bH}^{\hbox{\footnotesize{max}}}_{\ell+1}-
w(\ell){\bH}^{\hbox{\footnotesize{max}}}_\ell=
u(a+\ell+1)-u(a+\ell)-(u(1-c)+v).
$$
Let us consider
$$
\Delta^{\bH}_{\ell+1}=w(\ell+1)
\left({\bH}^{\hbox{\footnotesize{max}}}_{\ell+1}-{\bH}_{\ell+1}\right)
-w(\ell)\left({\bH}^{\hbox{\footnotesize{max}}}_\ell-{\bH}_\ell\right).
$$
If in the transition $\ell\to \ell+1$ the number of classes changes from 
$k\to k+1$ one gets that
$$
\Delta^{\bH}_{\ell+1}=0.
$$
If in the transition $\ell\to \ell+1$ the number of classes is preserved, 
say $k$, and the class $j^*$ adds in one unit we get
$$
\Delta^{\bH}_{\ell+1}=u(n^\ell_{j^*}-c+1) -u(n^\ell_{j^*}-c) - (u(1-c)+v).
$$
We combine (\ref{constct1}) with the third condition in (\ref{constct}), 
to deduce that when the transition $\ell$ to $\ell+1$ preserves the number of 
classes then $\Delta^{\bH}_{\ell+1}> 0$. Hence, the results for the variation
of the weighted differences of the maximal entropy to the entropy hold for this
class of entropies (\ref{genent}).

\medskip

\noindent Finally, let us see what one requires to have
$$
w(\ell+1){\bH}_{\ell+1}-w(\ell){\bH}_{\ell}=(u(a+\ell+1)-u(a+\ell))
-(u(n^\ell_{j^*}-c+1) -u(n^\ell_{j^*}-c))\ge 0. 
$$
Since from (\ref{constct}) we have $a\ge -c$ 
and so the unique new condition is
$$
u(n+a+1)-u(n+a)\ge u(m-c+1)-u(m-c) \hbox{ for } n\ge m, 
$$
which is satisfied for both, the PDP and the frequentist case.

\bigskip

\noindent {\bf Acknowledgments.} This work was supported by the Center for 
Mathematical Modeling ANID Basal PIA program FB210005.
In addition, we would like to
thank the reviewer for their careful reading and valuable comments and 
suggestions,
which helped to clarify and improve the presentation of the article.

\end{document}